\documentclass[12pt]{article}

\usepackage{amsmath,amssymb,amsthm}

\usepackage{verbatim}

\usepackage{mathtools}

\newcommand{\R}{\mathbb{R}}

\newcommand{\Hy}{\mathbb{H}}

\def\Xint#1{\mathchoice
{\XXint\displaystyle\textstyle{#1}}%
{\XXint\textstyle\scriptstyle{#1}}%
{\XXint\scriptstyle\scriptscriptstyle{#1}}%
{\XXint\scriptscriptstyle\scriptscriptstyle{#1}}%
\!\int}
\def\XXint#1#2#3{{\setbox0=\hbox{$#1{#2#3}{\int}$ }
\vcenter{\hbox{$#2#3$ }}\kern-.6\wd0}}

\def\dashint{\Xint-}

\author{Brian Allen}

\date{Spring 2018}

\title{Stability of the PMT and RPI for Asymptotically Hyperbolic Manifolds Foliated by IMCF}

\newtheorem{Thm}{Theorem}[section]
\newtheorem{Cor}[Thm]{Corollary}

\newtheorem{Prop}[Thm]{Proposition}
\newtheorem{Lem}[Thm]{Lemma}

\newtheorem{Def}[Thm]{Definition}

\begin{document}

\maketitle

\begin{abstract}
We study the stability of the Positive Mass Theorem (PMT) and the Riemannian Penrose Inequality (RPI) in the case where a region of an asymptotically hyperbolic manifold $M^3$ can be foliated by a smooth solution of Inverse Mean Curvature Flow (IMCF) which is uniformly controlled. We consider a sequence of regions of asymptotically hyperbolic manifolds $U_T^i\subset M_i^3$, foliated by a smooth solution to IMCF which is uniformly controlled, and if $\partial U_T^i = \Sigma_0^i \cup \Sigma_T^i$ and $m_H(\Sigma_T^i) \rightarrow 0$ then $U_T^i$ converges to a topological annulus portion of hyperbolic space with respect to $L^2$ metric convergence. If instead $m_H(\Sigma_T^i)-m_H(\Sigma_0^i) \rightarrow 0$ and $m_H(\Sigma_T^i) \rightarrow m >0$ then we show that $U_T^i$ converges to a topological annulus portion of the Anti-de~Sitter Schwarzschild metric with respect to $L^2$ metric convergence.
\end{abstract}

\section{Introduction}\label{sec:intro}
If we consider a complete, asymptotically Hyperbolic manifold, $M^3$, with scalar curvature $R \ge -6$ then the Positive Mass Theorem (PMT) in the asymptotically hyperbolic case says that $M^3$ has positive mass. The rigidity statement says that if the mass is $0$ then $M$ is isometric to hyperbolic space. Similarly, the Riemannian Penrose Inequality in the asymptotically hyperbolic case says that if $\partial M$ consists of a surface $\Sigma_0$ with $H=2$ then
\begin{align}
 m_{ADM}(M) \ge \sqrt{\frac{|\Sigma_0|}{16 \pi}}
\end{align}
where $|\Sigma_0|$ is the area of $\Sigma_0$. 
In the case of equality $M$ is isometric to the Anti-de~Sitter Schwarzschild metric. In this paper we are concerned with the stability of these two rigidity statements in the case where we can foliate a region of $M$ by a smooth solution of Inverse Mean Curvature Flow (IMCF) that is uniformly controlled.

The stability problem for the PMT  and RPI in the asymptotically flat case has been studied by Lee \cite{L}, Lee and Sormani \cite{LS1,LS2}, Huang, Lee and Sormani \cite{HLS},  LeFloch and Sormani \cite{LeS},  Finster \cite{F}, Finster and Bray \cite{BF},  Finster and Kath  \cite{FK}, and by Corvino \cite{C} (See \cite{BA2} for a further discussion of these results).  In addition, the author \cite{BA2} has recently shown how to use  IMCF to prove stability of the PMT and RPI for asymptotically flat manifolds which are foliated by IMCF, obtaining $L^2$ metric convergence to Euclidean space or the Schwarzschild metric. The goal of this paper is to extend the results of \cite{BA2} to the asymptotically hyperbolic case.

Previously, Dahl, Gicquad and Sakovich have adapted the stability of the asymptotically flat PMT results of Lee \cite{L} to manifolds which are conformal to hyperbolic space outside a compact set, overcoming some new issues arising in the asymptotically hyperbolic case. The question of what can happen inside this compact set is still open and recently Sakovich and Sormani \cite{SS} extended the results of Lee and Sormani \cite{LS2} to the PMT in the asymptotically hyperbolic case for rotationally symmetric manifolds, showing intrinsic flat convergence to an annular region of Hyperbolic space.

With the application of stability in mind, one should think of IMCF as providing good coordinates for our manifold $M_j^3$ which is analagous to the rotationally symmetric coordinates used in \cite{LS2,LS1,SS}. These coordinates allow us to express the metric in an advantageous form, see equation \eqref{metricexpression}, which leads to arguing that the metric converges to hyperbolic space or the Anti-de~Sitter Schwarzschild metric using estimates on geometric quantities under IMCF. In \cite{BA4} the author shows existence of IMCF coordinates on metrics which are conformal to warped products with particular analytic assumptions on the warping and conformal factor which when combined with the results of this paper immediately imply stability results.

In \cite{HI}, Huisken and Ilmanen show how to use weak solutions of IMCF in order to prove the PMT for asymptotically flat Riemanian manifolds as well as the RPI in the case of a connected boundary. Later Neves \cite{AN}, and Hung and Wang \cite{HW} showed that IMCF does not have strong enough convergence properties to extend the proof of Huisken and Ilmanen to the asymptotically hyperbolic case. Despite this fact, IMCF has still been used to prove important geometric inequalities for asymptotically hyperbolic manifolds by Brendle \cite{Br}, Brendle, Hung and Wang \cite{BHW}, and de Lima and Girao \cite{DG}, to name a few. In this work we will be able to use IMCF in the asymptotically hyperbolic setting to prove a stability result by stating stability in terms of the Hawking mass, as done in the author's previous work \cite{BA2}.

We remember that IMCF is defined for surfaces $\Sigma^n \subset M^{n+1}$ evolving through a one parameter family of embeddings $F: \Sigma \times [0,T] \rightarrow M$, $F$ satisfying inverse mean curvature flow,

\begin{equation}\label{IMCF}
\begin{cases}
\frac{\partial F}{\partial t}(p,t) = \frac{\nu(p,t)}{H(p,t)}  &\text{ for } (p,t) \in \Sigma \times [0,T),
\\ F(p,0) = \Sigma_0  &\text{ for } p \in \Sigma ,
\end{cases}
\end{equation}
where $H$ is the mean curvature of $\Sigma_t := F_t(\Sigma)$ and $\nu$ is the outward pointing normal vector. The outward pointing normal vector will be well defined in our case since we have in mind considering $M^3$ to by asymptotically hyperbolic manifolds with one end.

For a glimpse of long time existence and asymptotic analysis results for smooth IMCF in various ambient manifolds see  \cite{BA,BA4,QD, CG1,CG2,S,S2,U,Z}. For our purposes here the results of Scheuer \cite{S, S2} and the author \cite{BA4} are particularly significant. In \cite{S, S2} Scheuer gives long time existence and asymptotic analysis results for rotationally symmetric metrics with non-positive radial curvature and later generalizing to warped products. The asymptotic results therein imply that after rescaling you will find that $\tilde{\Sigma}_t$ converges to a $C^{2,\alpha}$ hypersurface but you cannot conclude it is a sphere, as expected for aymptotically hyperbolic spaces. It is for this reason that we will try to impose the mildest conditions possible on a long time solution of IMCF to achieve stability when stating Corollaries \ref{PMTCOR} and \ref{RPICOR}. In particular the asymptotic conditions assumed here are compatible with the conditions imposed in the author's work \cite{BA4} where long time existence of IMCF is shown on metrics conformal to warped products.

Now the class of regions of asymptotically Hyperbolic manifolds to which we will be proving stability of the PMT and RPI is defined. Note that we have in mind foliations of asymptotically hyperbolic manifolds by IMCF but we will not need to assume that $M$ is asymptotically hyperbolic manifold to state our desired theorems. Afterwards, when we state Corollaries \ref{PMTCOR} and \ref{RPICOR} in the case where we assume a long time solution of IMCF then we will define asymptotically hyperbolic manifolds as well as other necessary definitions for this setting.

As far as notation is concerned, if we have $\Sigma^2$ a surface in a Riemannian manifold, $M^3$, we will denote the induced metric, mean curvature, second fundamental form, principal curvatures, Gauss curvature, area, Hawking mass and Neummann isoperimetric constant as $g$, $H$, $A$, $\lambda_i$, $K$, $|\Sigma|$, $m_H(\Sigma)$, $IN_1(\Sigma)$, respectively. We will denote the Ricci curvature, scalar curvature, and sectional curvature tangent to $\Sigma$ as $Rc$, $R$, $K_{12}$, respectively.

\begin{Def} \label{IMCFClass} Define the class of manifolds with boundary foliated by IMCF as follows
\begin{align*}
\mathcal{M}_{r_0,H_0,I_0}^{T,H_1,A_1}:=\{ U_T \subset M,  R \ge -6|&\exists \Sigma \subset M \text{compact, connected surface such that } 
\\& IN_1(\Sigma) \ge I_0, m_H(\Sigma) \ge 0 \text{,and } |\Sigma|=4\pi r_0^2. 
\\ &\exists \Sigma_t \text{ smooth solution to IMCF, s.t. }\Sigma_0=\Sigma,
\\& H_0 \le H(x,t) \le H_1 < \infty, |A|(x,t) \le A_1 \text{ for } t \in [0,T],
\\&\text{and } U_T = \{x\in\Sigma_t: t \in [0,T]\} \}
\end{align*}
where $0 < H_0 < H_1 < \infty$, $0 < I_0,A_1,r_0 < \infty$ and $0 < T < \infty$.
\end{Def}

\textbf{Note:} The regions defined above could be obtained from a weak solution of IMCF (see \cite{HI}) where the region $U_T$ lies in between two jump times of the weak flow since Huisken and Ilmanen showed that weak solutions are smooth between jump times. Also, $U_T$ could be a region of a manifold which is conformal to a warped product which falls under the assumptions of the author's work \cite{BA4}. The author also believes that long time existence of IMCF starting from a coordinate sphere in a nearly rotationally symmetric manifold should be true which would imply stability for these manifolds and would lead to an extension of the work by Sakovich and Sormani \cite{SS}.

Before we state the stability theorems we define some metrics on $\Sigma \times [0,T]$ that will be used throughout this document:
\begin{align}
\bar{g} &= \frac{1}{4}\left(1+ \frac{e^{-t}}{r_0^2} \right)^{-1}dt^2 + r_0^2e^t \sigma,
\\ g_{AdSS} &= \frac{1}{4}\left (\frac{1}{r_0^2}e^{-t} - \frac{2}{r_0^3} m e^{-3t/2}+1 \right )^{-1} dt^2 + r_0^2e^t \sigma,
\\ \hat{g}^i&=\frac{1}{H(x,t)^2}dt^2 + g^i(x,t), \label{metricexpression}
\end{align}
where $\sigma$ is the round metric on $\Sigma$ and $g^i(x,t)$ is the metric on $\Sigma_t^i$. The first metric is the metric of hyperbolic space $\Hy^3$, the second is the Anti-de~Sitter Schwarzschild metric and the third is the metric on $U_T^i$ with respect to the foliation by IMCF. The first two can be verified by defining $s = r_0 e^{t/2}$ then $ds^2 = \frac{r_0^2}{4}e^t dt^2$, $\hat{g} = \frac{1}{1+s^2}ds^2 + s^2 d\sigma$ which is isometric to Hyperbolic space and $g_{AdSS} = \frac{1}{1 - \frac{2m}{s}+s^2}ds^2 + s^2 d\sigma$ which is isometric to the Anti-de~Sitter Schwarzschild space.

\begin{Thm}\label{SPMT}
Let $U_{T}^i \subset M_i^3$ be a sequence such that $U_{T}^i\subset \mathcal{M}_{r_0,H_0,I_0}^{T,H_1,A_1}$ and $m_H(\Sigma_{T}^i) \rightarrow 0$ as $i \rightarrow \infty$.  

If we assume one of the following conditions,
\begin{align}
& \exists I > 0 \text{ so that } K_{12}^i \ge -1 \text{ on } \Sigma_0 \text{ and diam }(\Sigma_0^i) \le D \text{ } \forall i \ge I \label{cond 1},
\\&  \exists  [a,b]\subset [0,T] \text{ s.t. } \| Rc^i(\nu,\nu)\|_{W^{1,2}(\Sigma\times [a,b])} \le C \text{ and }
\\& \hspace{1cm} \text{diam}(\Sigma_t^i) \le D \text{ } \forall i \text{, }t\in [a,b],
\end{align}
where $W^{1,2}(\Sigma\times [a,b])$ is defined with respect to $\delta$, then
\begin{align}
\hat{g}^i \rightarrow \bar{g}
\end{align}
in $L^2$ with respect to $\bar{g}$.
\end{Thm}

\begin{Thm}\label{SRPI}
Let $U_{T}^i \subset M_i^3$ be a sequence such that $U_{T}^i\subset \mathcal{M}_{r_0,H_0,I_0}^{T,H_1,A_1}$, $m_H(\Sigma_{T}^i)- m_H(\Sigma_0^i) \rightarrow 0$ and $m_H(\Sigma_0) \rightarrow m > 0$ as $i \rightarrow \infty$. 

If we assume the following condition,
\begin{align}
&  \exists  [a,b]\subset [0,T] \text{ s.t. } \| Rc^i(\nu,\nu)\|_{W^{1,2}(\Sigma\times [a,b])} \le C \text{ and }
\\& \hspace{1cm} \text{diam}(\Sigma_t^i) \le D \text{ } \forall i \text{, }t\in [a,b],
\end{align}
where $W^{1,2}(\Sigma\times [a,b])$ is defined with respect to $\delta$, then
\begin{align}
\hat{g}^i \rightarrow g_{AdSS}
\end{align}
in $L^2$ with respect to $g_{AdSS}$.
\end{Thm} 

\textbf{Note:} We will be able to show that for $i$ large enough $\Sigma_0^i$ is topologically a sphere and hence $U_T^i$ will be diffeomorphic to an annular region in $\R^3$. This motivates considering the metrics $\hat{g}, \delta, g_{ADSS}$ on $\Sigma\times [0,T]$.

The theorems above give stability for compact regions which are foliated by uniformly controlled solutions of IMCF. Now one can ask if this stability holds on an asymptotically hyperbolic end of a manifold which is completely foliated by IMCF. Can you satisfy the assumptions of Theorems \ref{SPMT} and \ref{SRPI} by taking advantage of the asymptotically hyperbolic end? We now move to answer this question by first stating the definition of an asymptotically hyperbolic manifold, as well as, what it means to be a uniformly asymptotically hyperbolic sequence of manifolds.

\begin{Def}\label{AH}
We say a complete, Riemannian manifold $(M^3,g)$ is an asymptotically hyperbolic manifold if there exists $K \subset M$, compact, so that $M \setminus K$ is diffeomorphic to $\R^3\setminus \overline{B}(0,1)$ and so that if we define the tensor,
\begin{align}
E=g - g_{\Hy},
\end{align}
where $g_{\Hy} = dr^2 + \sinh^2(|x|) d \sigma$, we have that
\begin{align}
|E_{ij}|+|E_{ij,k}|+|E_{ij,kl}|+|E_{ij,klm}|\le C e^{-\alpha |x|},\label{DerMetCond}
\end{align}
as $|x| \rightarrow \infty$ where $\alpha > 0$ and the norm/derivatives are taken with respect to $\delta$. If $\partial M \not = \emptyset$ then we require $\partial M$ to be an surface with $H = 2$.

We say a sequence of asymptotically hyperbolic manifolds $M_j=(M,g_j)$ is uniformly asymptotically hyperbolic if the constant $C$ above can be chosen uniformly for the sequence.
\end{Def}

In the author's paper \cite{BA2} for asymptotically flat manifolds it was perfectly reasonable to assume that the IMCF coordinates were asymptotically flat coordinates since we expect good convergence of long time solutions of IMCF in that case. In the hyperbolic case one does not want to make such a restrictive assumption and so we define now what it means for the IMCF to be uniformly compatible with the asymptotically hyperbolic coordinates.

\begin{Def}\label{IMCFcompatibleCoordsDef}
Let $(M^3, g)$ be an asymptotically hyperbolic manifold with induced radial coordinate $r \in [r_0,\infty)$ on the exterior region $M\setminus K$ where $K$ is a compact set. If $M\setminus K$ is also foliated by a solution to IMCF, $\Sigma_t$, then we say that the IMCF coordinates are compatible with the asymptotically hyperbolic coordinates if for each $t \in [0,\infty)$ we have,
\begin{align}
C_1 t  &\le r(x,t) \le  C_2 t, \label{rtBothBigAssumption}
\end{align}
where we let $r(x,t)$ represent the $r$ coordinate of $\Sigma_t$.
In addition, $\Sigma_t$ can be expressed as a graph over the coordinate ball $B_{r_0}=\{r=r_0\}$ with graph function $f$ so that,
\begin{align}
|\nabla^{S^2} f|\le C_3 \text{, for all } t \in [0,\infty).\label{UniformGraphAssumption}
\end{align}
We say a sequence of asymptotically hyperbolic manifolds $M_j = (M,g_j)$ is uniformly compatible with the IMCF coordinates if the constants $C_1, C_2, C_3$ can be chosen uniformly for the sequence.
\end{Def}

For the purposes of proving stability of the PMT or RPI using IMCF it is convenient to use the Hawking mass because of the fact that it is monotone along IMCF. For this reason we define a type of mass at infinity using the Hawking mass which is dependent on the initial hypersurface $\Sigma_0$. This mass at infinity will be used to state the stability corollaries below.

\begin{Def}\label{MassAtInfinityDef}
Let $\Sigma_t$ be a smooth solution to \eqref{IMCF} which exists for all $t \in [0,\infty)$ and so that $m_H(\Sigma_0) \ge 0$. Then we define the Hawking mass at infinity to be,
\begin{align}
m_H(\Sigma_{\infty}^i) = \displaystyle\lim_{t \rightarrow \infty} m_H(\Sigma_t^i).
\end{align}
\end{Def}
\textbf{Note:} It was shown by Andre Neves \cite{AN} that $m_H(\Sigma_{\infty}^i)$ can be bigger than or smaller than the true mass of an asymptotically hyperbolic manifold (for definition see \cite{CH, DGS, AN, W}) but nonetheless we will see that we can show stability if $m_H(\Sigma_{\infty}^i)\rightarrow 0$ in the corollaries below. In particular, if one chooses a smooth solution to IMCF which does converge to the true mass for each $i$ then the stability results will apply to this sequence.

\begin{Cor}\label{PMTCOR}
Assume for all $M_i$ the smooth solution of IMCF starting at $\Sigma_0$ exists for all time, so that for all $T \in (0, \infty)$, $U_{T,i}\subset \mathcal{M}_{r_0,H_0^T,I_0}^{T,H_1,A_1}$. In addition, assume that $\displaystyle m_H(\Sigma_{\infty}^i) \rightarrow 0$ as $i \rightarrow \infty$ and that $M_i$ are uniformly asymptotically hyperbolic and uniformly compatible with the IMCF coordinates then,
\begin{align}
\hat{g}^i \rightarrow \bar{g},
\end{align}
on $\Sigma\times[0,T]$ in $L^2$ with respect to $\bar{g}$.
\end{Cor}

\begin{Cor}\label{RPICOR}
Assume that for all $M_i$ the smooth solution of IMCF starting at $\Sigma_0$ exists for all time, so that $U_{T,i}\subset \mathcal{M}_{r_0,H_0^T,I_0}^{T,H_1,A_1}$ for all $T \in (0, \infty)$. In addition, assume that $m_H(\Sigma_{\infty}^i)- m_H(\Sigma_{0}^i) \rightarrow 0$, $m_H(\Sigma_{\infty}^i)\rightarrow m > 0$ as $i \rightarrow \infty$, and $M_i$ are uniformly asymptotically hyperbolic and uniformly compatible with the IMCF coordinates then,
\begin{align}
\hat{g}^i \rightarrow g_{AdSS},
\end{align}
on $\Sigma\times[0,T]$ in $L^2$ with respect to $g_{AdSS}$.
\end{Cor}

\textbf{Note:} The author has proved \cite{BA4} long time existence results for IMCF in manifolds which are conformal to warped products with assumptions on the warping function and conformal factor. These long time existence results apply to asymptotically hyperbolic manifolds and when combined with Corollary \ref{PMTCOR} and Corollary \ref{RPICOR} should imply stability results for these manifolds.

In Section 2 we will use IMCF to get important estimates of the metric $\hat{g}$ in the foliated region $U_T^i\subset M_i$. The crucial estimates come from the calculation of the monotonicity of the Hawking mass in Lemma \ref{CrucialEstimate} which lead to integrals of geometric quantities converging to zero in Corollary \ref{GoToZero}. At the end of this section, the exact diffeomorphism that we are using to induce coordinates on the regions $U_T^i$ is discussed in Proposition \ref{AvgHEst} which is used implicitly throughout the rest of the paper.  

In Section 3 we use the estimates of the previous section to show convergence of $\hat{g}$ to a warped product $g_3^i(x,t) = \frac{1}{4}\left(1+ \frac{e^{-t}}{r_0^2} \right)^{-1}dt^2 + r_0^2e^t g^i(x,0)$ or $g_3^i(x,t) = \frac{1}{4}\left (\frac{1}{r_0^2}e^{-t} - \frac{2}{r_0^3} m e^{-3t/2}+1 \right )^{-1}dt^2 + r_0^2e^t g^i(x,0)$. This is done by showing convergence of $\hat{g}$ to simpler metrics, successively, until we get to $g_3^i$ and combining this chain of estimates by the triangle inequality. This proves Theorems \ref{SPMT} and \ref{SRPI} in the rotationally symmetric case since then we know that $\Sigma_t$ can be taken to be spheres.

In Section 4 we complete the proofs of Theorems \ref{SPMT} and \ref{SRPI} by showing convergence of $g_3^i$ to $\bar{g}$ or $g_{AdSS}$. For this we need something besides IMCF to complete the job which is where the assumptions of Theorem \ref{SPMT} and Theorem \ref{SRPI} come into play. These assumptions and the results that follow are combined with the rigidity result of Petersen and Wei \cite{PW}, Theorem \ref{rigidity}, in order to improve from $L^2$ curvature convergence results to $L^2$ metric convergence, which completes the proof of Theorems \ref{SPMT} and \ref{SRPI}.

\section{Estimates for Asymptotically Hyperbolic Manifolds Foliated by IMCF}\label{Sect-Est}

We start by obtaining some simple but useful estimates where it will be important to remember the definition of the Hawking mass defined for a hypersurface $\Sigma^2 \subset M^3$, where $M^3$ is an asymptotically Hyperbolic manifold,
\begin{align}
m_H(\Sigma) = \sqrt{\frac{|\Sigma|}{(16\pi)^3}} \left (16 \pi - \int_{\Sigma} H^2-4 d \mu \right ).
\end{align}

\begin{Lem} \label{naiveEstimate}
Let $\Sigma^2 \subset M^3$ be a hypersurface and $\Sigma_t$ its corresponding solution of IMCF. If $m_1 \le m_H(\Sigma_t) \le m_2$ then 
\begin{align}
|\Sigma_t|&=|\Sigma_0| e^t
\\16 \pi \left (1 - \sqrt{\frac{16 \pi}{|\Sigma_0|}}m_2e^{-t/2}  \right ) &\le \int_{\Sigma_t} H^2-4 d \mu \le 16 \pi \left (1 - \sqrt{\frac{16 \pi}{|\Sigma_0|}}m_1e^{-t/2}  \right ) \label{Eq-NaiveAvg}
\\ \frac{16 \pi}{|\Sigma_0|} \left (1 - \sqrt{\frac{16 \pi}{|\Sigma_0|}}m_2e^{-t/2}  \right )e^{-t} &\le \dashint_{\Sigma_t} H^2-4 d \mu \le  \frac{16 \pi}{|\Sigma_0|} \left (1 - \sqrt{\frac{16 \pi}{|\Sigma_0|}}m_1e^{-t/2}  \right )e^{-t}\label{AvgHEst}
\end{align}
where $|\Sigma_t|$ is the $n$-dimensional area of $\Sigma$ and $V(\Sigma_t)$ is the $n+1$-dimensional enclosed volume.

Hence if $m_{H}(\Sigma_T^i) \rightarrow 0$ then ,
\begin{align}
\bar{H^2}_i(t):=\dashint_{\Sigma_t^i} H_i^2 d \mu \rightarrow  \frac{4}{r_0}e^{-t}+4,\label{unifAvgHEst1}
\end{align}
for every $t\in [0,T]$.

If $m_{H}(\Sigma_T^i)-m_{H}(\Sigma_0^i) \rightarrow 0$ and $m_{H}(\Sigma_0^i) \rightarrow m > 0$ then,
\begin{align}
\bar{H^2}_i(t):=\dashint_{\Sigma_t^i} H_i^2 d \mu\rightarrow \frac{4}{r_0^2} \left(1-\frac{2}{r_0}m e^{-t/2}\right)e^{-t} +4, \label{unifAvgHEst2}
\end{align}
for every $t\in [0,T]$.
\end{Lem}

\begin{proof}
The last two estimates follow directly from the definition of the Hawking mass and the first estimate is standard for IMCF. The equations \eqref{unifAvgHEst1} and \eqref{unifAvgHEst2} follow from \eqref{AvgHEst} and the assumption on the Hawking mass along the sequence.
\end{proof}

\begin{Lem}\label{dtintEstimate}
For any solution of IMCF we have the following formula
\begin{align}
\frac{d}{dt} \int_{\Sigma_t} H^2-4 d\mu =  \frac{(16 \pi)^{3/2}}{|\Sigma_t|^{1/2}} \left (\frac{1}{2} m_H(\Sigma_t) - \frac{d}{dt}m_H(\Sigma_t) \right )
\end{align}
So if we assume that $m_H(\Sigma_t^i) \rightarrow 0$ as $i \rightarrow \infty$ then we have for a.e. $t \in [0,T]$ that ,
\begin{align}
\frac{d}{dt} \int_{\Sigma_t^i} H^2-4 d\mu \rightarrow 0.\label{Eq-dtH^2PMT}
\end{align}

If we assume that $m_H(\Sigma_T^i) - m_H(\Sigma_0^i) \rightarrow 0$ and $m_H(\Sigma_t^i)\rightarrow m > 0$ as $ i \rightarrow \infty$ then we have that, 
\begin{align}
\frac{d}{dt} \int_{\Sigma_t^i} H^2-4 d\mu \rightarrow \frac{16 \pi}{r_0}me^{-t/2}. \label{Eq-dtH^2RPI}
\end{align}
\end{Lem}

\begin{proof}
By using the formula for the Hawking mass we can compute that
\begin{align}
\frac{d}{dt}m_H(\Sigma_t)&= \frac{1}{2}m_H(\Sigma_t) - \sqrt{\frac{|\Sigma_t|}{(16\pi)^3}} \frac{d}{dt}\int_{\Sigma} H^2-4 d \mu \label{Eq-dtInt}
\end{align}
Rearranging this equation by solving for $\frac{d}{dt}\int_{\Sigma} H^2-4 d \mu $ we find the first formula in the statement of the lemma.

By Geroch monotonicity we know that $\frac{d}{dt}m_H(\Sigma_t) \ge 0$ and so if $m_H(\Sigma_t^i)\rightarrow 0$ as $i \rightarrow \infty$ then we must have that $ \frac{d}{dt}m_H(\Sigma_t^i) \rightarrow 0$ for almost every $t \in [0,T]$. Combining with \eqref{Eq-dtInt} shows that $\frac{d}{dt}\int_{\Sigma_t^i}H_i^2-4 d\mu \rightarrow 0$ for almost every $t \in [0,T]$.

If $m_H(\Sigma_T^i) - m_H(\Sigma_0^i) \rightarrow 0$ as $ i \rightarrow \infty$ then we have that $\int_0^T \frac{d}{dt}m_H(\Sigma_t)dt \rightarrow 0$ and so by Geroch monotonicity we must have that $\frac{d}{dt}m_H(\Sigma_t) \rightarrow 0$ for almost every $t \in [0,T]$. Then by combining with the assumption that $m_H(\Sigma_t^i)\rightarrow m$ as $ i \rightarrow \infty$ we get the desired result in this case.

\end{proof}

Now we obtain the following crucial estimate.

\begin{Lem}\label{CrucialEstimate} Let $\Sigma^2\subset M^3$ be a compact, connected surface with corresponding solution to IMCF $\Sigma_t$ then,
\begin{align}
m_H(\Sigma_t)\left (\frac{(16 \pi)^{3/2}}{2|\Sigma_t|^{1/2}}\right ) \ge \frac{d}{dt} \int_{\Sigma_t} H^2-4 d\mu + \int_{\Sigma_t}2\frac{|\nabla H|^2}{H^2} +\frac{1}{2}(\lambda_1-\lambda_2)^2 + R+6 d\mu, \label{Eq-Crux}
\end{align}
which can be rewritten and integrated to find,
\begin{align}
m_H(\Sigma_T)-m_H(\Sigma_0)  \ge \int_0^T \frac{|\Sigma_t|^{1/2}}{(16 \pi)^{3/2}} \left (\int_{\Sigma_t}2\frac{|\nabla H|^2}{H^2} +\frac{1}{2}(\lambda_1-\lambda_2)^2 + R+6 d\mu \right)dt.
\end{align}
\end{Lem}

\begin{proof}
We will use the following facts in the derivation below where $R$ is the scalar curvature of $M$ and $K$ is the Gauss curvature of $\Sigma_t$.

\begin{align}
-Rc(\nu,\nu) &= -\frac{R}{2} +K - \frac{1}{2}(H^2 - |A|^2)
\\ |A|^2 &= \frac{1}{2}H^2 + \frac{1}{2} (\lambda_1-\lambda_2)^2
\\ \int_{\Sigma_t} K d\mu_t &= 2 \pi \chi(\Sigma_t)
\end{align}
which follow from the Gauss equations, the definition of $|A|^2$ and the Gauss-Bonnet theorem (where $\lambda_i$ are the principal curvatures of $\Sigma$). We will use these equations below.

Now we compute the time derivative of $\int_{\Sigma_t}H^2 d \mu$

\begin{align}
\frac{d}{dt} \int_{\Sigma_t} H^2-4 d \mu_t &= \int_{\Sigma_t} 2 H \frac{\partial H}{\partial t} + H^2 -4d \mu_t
\\&=\int_{\Sigma_t} -2H \Delta \left ( \frac{1}{H} \right ) - 2 |A|^2 -2  Rc(\nu,\nu) + H^2-4 d \mu_t \label{Eq-dtFirst}
\\ &=\int_{\Sigma_t}-2\frac{|\nabla H|^2}{H^2} -|A|^2 - R + 2K  -4d \mu_t
\\&= 4 \pi  \chi(\Sigma_t) +\int_{\Sigma_t}-2\frac{|\nabla H|^2}{H^2} - \frac{1}{2}(\lambda_1-\lambda_2)^2 - R  -\frac{1}{2}H^2 -4d \mu_t
\\&\le m_H(\Sigma_t) \frac{(16 \pi)^{3/2}}{2|\Sigma_t|^{1/2}}+ \int_{\Sigma_t}-2\frac{|\nabla H|^2}{H^2} - \frac{1}{2}(\lambda_1-\lambda_2)^2 - R -6d \mu_t \label{Eq-lastCruc}
\end{align}
where we are using that $\chi(\Sigma_t)\le 2$ for compact, connected surfaces.
Rearranging \eqref{Eq-lastCruc} we find that
\begin{align}
m_H(\Sigma_t) \frac{(16 \pi)^{3/2}}{|\Sigma_t|^{1/2}} \ge \frac{d}{dt} \int_{\Sigma_t} H^2-4 d\mu + \int_{\Sigma_t}2\frac{|\nabla H|^2}{H^2} +\frac{1}{2}(\lambda_1-\lambda_2)^2 + R+6 d\mu
\end{align}
Now by combining with Lemma \ref{dtintEstimate} we find 
\begin{align}
\frac{d}{dt}m_H(\Sigma_t)  \ge \frac{|\Sigma_t|^{1/2}}{(16 \pi)^{3/2}} \int_{\Sigma_t}2\frac{|\nabla H|^2}{H^2} +\frac{1}{2}(\lambda_1-\lambda_2)^2 + R+6 d\mu
\end{align}
and then by integrating both sides from $0$ to $T$ we find the desired estimate.
\end{proof}

By combining Lemma \ref{CrucialEstimate} with Lemma \ref{dtintEstimate} we are able to deduce the crucial estimates below which we will show leads to a stability of positive mass theorem.

\begin{Cor} \label{GoToZero}Let $\Sigma^i\subset M^i$ be a compact, connected surface with corresponding solution to IMCF $\Sigma_t^i$. If $m_H(\Sigma_0)\ge0$ and $m_H(\Sigma^i_T) \rightarrow 0$  then for almost every $t \in [0,T]$ we have that,
\begin{align}
&\int_{\Sigma_t^i} \frac{|\nabla H_i|^2}{H_i^2}d \mu \rightarrow 0, \hspace{1.2 cm} \int_{\Sigma_t^i} (\lambda_1^i-\lambda_2^i)^2d \mu \rightarrow 0,\hspace{.8 cm} \int_{\Sigma_t^i} R^i+6 d \mu \rightarrow 0,
\\&\int_{\Sigma_t^i} Rc^i(\nu,\nu)+2d \mu \rightarrow 0, \hspace{.1 cm} \int_{\Sigma_t^i} K_{12}^i+1d \mu \rightarrow 0, \hspace{1 cm} \int_{\Sigma_t^i} H_i^2-4 d\mu\rightarrow 16\pi,
\\&\int_{\Sigma_t^i} |A|_i^2-2 d \mu \rightarrow 8 \pi,\hspace{.5 cm} \int_{\Sigma_t^i} \lambda_1^i\lambda_2^i-1 d \mu \rightarrow 4\pi, \hspace{1 cm} \chi(\Sigma_t^i) \rightarrow 2,
\end{align}
as $i \rightarrow \infty$ where $K_{12}$ is the ambient sectional curvature tangent to $\Sigma_t$. We also find that $\Sigma_t^i$ must eventually become topologically a sphere. 

If $\left (m_H(\Sigma^i_T)-m_H(\Sigma^i_0) \right ) \rightarrow 0$ where $m_H(\Sigma_0) \rightarrow m > 0$ then the first three integrals listed above tend to zero and for almost every $t \in [0,T]$ we have that,
\begin{align}
 &\int_{\Sigma_t^i} H_i^2-4 d\mu\rightarrow 16 \pi \left (1 - \sqrt{\frac{16 \pi}{|\Sigma_0|}}me^{-t/2}  \right ),
 \hspace{0.2 cm} \int_{\Sigma_t^i} |A|_i^2-2 d \mu \rightarrow 8 \pi \left (1 - \sqrt{\frac{16 \pi}{|\Sigma_0|}}me^{-t/2}  \right ), \hspace{0.5 cm}
 \\& \int_{\Sigma_t^i} \lambda_1^i\lambda_2^i-1 d \mu \rightarrow 4 \pi \left (1 - \sqrt{\frac{16 \pi}{|\Sigma_0|}}me^{-t/2}  \right ),\hspace{0.5 cm}\int_{\Sigma_t^i}Rc^i(\nu,\nu)+2d\mu \rightarrow -\frac{8\pi}{r_0}m e^{-t/2},
 \\ &\int_{\Sigma_t^i}K_{12}^i+1d\mu \rightarrow \frac{8\pi}{r_0}m e^{-t/2},\hspace{3 cm}\chi(\Sigma_t^i) \rightarrow 2.
\end{align}
We also find that $\Sigma_t^i$ must eventually become topologically a sphere.
\end{Cor}

\begin{proof}
The first three integrals converge to $0$ by Lemma \ref{CrucialEstimate} \eqref{Eq-Crux} so now we will show how to deduce the last three. Using the calculation in \ref{CrucialEstimate}   we can rewrite \eqref{Eq-dtFirst} as
\begin{align}
\frac{d}{dt} \int_{\Sigma_t^i} H_i^2-4 d \mu_t &=\int_{\Sigma_t^i} -2\frac{|\nabla H_i|^2}{H_i^2} - (\lambda_1^i-\lambda_2^i)^2 -2  Rc^i(\nu,\nu)  -4d \mu_t
\end{align}
which implies that $\int_{\Sigma_t^i}Rc^i(\nu,\nu)+2d \mu \rightarrow 0$ for almost every $t \in [0,T]$ since every other integral in that expression $\rightarrow 0$ for almost every $t \in [0,T]$. Then we can write
\begin{align}
\int_{\Sigma_t^i} K_{12}^i+1 d\mu &=\int_{\Sigma_t^i} \frac{1}{2}(R^i- 2Rc^i(\nu,\nu)+2)  d \mu
\\&= \int_{\Sigma_t^i}  \frac{1}{2}(R^i+6)-(  Rc^i(\nu,\nu)+2)  d \mu
\end{align}
which implies that $\int_{\Sigma_t^i}K_{12}^i+1d \mu \rightarrow 0$ for almost every $t \in [0,T]$. Lemma \ref{naiveEstimate} equation \eqref{Eq-NaiveAvg} implies that $\int_{\Sigma_t^i} H_i^2-4 d\mu\rightarrow 16\pi$ and so if we write
\begin{align}
\int_{\Sigma_t^i} |A|^2_i-2d\mu &= \frac{1}{2}\int_{\Sigma_t^i} H_i^2-4 + (\lambda_1^i-\lambda_2^i)^2 d \mu \rightarrow 0.
\end{align}

Lastly we notice
\begin{align}
\int_{\Sigma_t^i} \lambda_1^i\lambda_2^i-1d \mu &=\int_{\Sigma_t^i} \frac{1}{2}(H_i^2 - |A|_i^2-2)
\\&=\int_{\Sigma_t^i} \frac{1}{2}(H_i^2-4) - \frac{1}{2}(|A|_i^2-2) \rightarrow 4\pi
\end{align}
and so
\begin{align}
2 \pi \chi(\Sigma_t^i) = \int_{\Sigma_t^i} K^i d\mu &= \int_{\Sigma_t^i}\lambda_1^i\lambda_2^i + K_{12}^i  d\mu
\\&=\int_{\Sigma_t^i}(\lambda_1^i\lambda_2^i - 1) + (K_{12}^i+1)  d\mu \rightarrow 4\pi.
\end{align}
Since $\chi(\Sigma_t^i)$ is discrete we see by the last convergence that $\Sigma_t^i$ must eventually become topologically a sphere. 

In the case where we assume $\left (m_H(\Sigma^i_T)-m_H(\Sigma^i_0) \right ) \rightarrow 0$ the convergence results follow similarly.
\end{proof}

Now we obtain an estimate which gives us weak convergence of $Rc^i(\nu,\nu)$ which will be used in Section \ref{sect-End}.

\begin{Lem}\label{WeakRicciEstimate}Let $\Sigma^i_0\subset M^3_i$ be a compact, connected surface with corresponding solution to IMCF $\Sigma_t^i$. Then if $\phi \in C^1(\Sigma\times (a,b))$  and $0\le a <b\le T$ we can compute the  estimate,
\begin{align}
\int_a^b\int_{\Sigma_t^i}& 2\phi Rc^i(\nu,\nu)d\mu d t=  \int_{\Sigma_a^i} \phi H_i^2 d\mu -  \int_{\Sigma_b^i} \phi H_i^2 d\mu 
\\&+ \int_a^b\int_{\Sigma_t^i}2\phi\frac{|\nabla H_i|^2}{H_i^2}-2\frac{\hat{g}^j(\nabla \phi, \nabla H_i)}{H_i} +\phi(H_i^2-2|A|_i^2) d\mu .
\end{align}

If $m_H(\Sigma^i_T) \rightarrow 0$ and $\Sigma_t$ satisfies the hypotheses of Proposition \ref{avgH} then the estimate above implies
\begin{align}
\int_a^b\int_{\Sigma_t}& \phi Rc^i(\nu,\nu)d\mu d t \rightarrow  \int_a^b \int_{\Sigma} -2r_0^2 e^t \phi d\sigma dt.
\end{align}
If $m_H(\Sigma^i_T)-m_H(\Sigma^i_0)  \rightarrow 0$, $m_H(\Sigma_T) \rightarrow m > 0$ and $\Sigma_t$ satisfies the hypotheses of Proposition \ref{avgH}  then the estimate above implies,
\begin{align}
\int_a^b\int_{\Sigma_t}& \phi Rc^i(\nu,\nu)d\mu d t \rightarrow \int_a^b\int_{\Sigma} -2\left(\frac{1 }{r_0}me^{-t/2} + r_0^2 e^t\right) \phi d\sigma dt.
\end{align}
\end{Lem}

\begin{proof}
Let $\phi \in C^1(\Sigma\times(a,b))$ and compute,
\begin{align}
&\frac{d}{dt} \int_{\Sigma_t^i}\phi H_i^2 d \mu_t = \int_{\Sigma_t^i} 2 \phi H_i \frac{\partial H_i}{\partial t} + \phi H_i^2 + \frac{\partial \phi}{\partial t}H_i^2d \mu
\\&=\int_{\Sigma_t^i} -2\phi H_i \Delta \left ( \frac{1}{H_i} \right ) - 2 \phi |A|_i^2 -2 \phi Rc^i(\nu,\nu) + \phi H^2_i + \frac{\partial \phi}{\partial t}H_i^2 d \mu
\\ &=\int_{\Sigma_t^i}-2\phi \frac{|\nabla H_i|^2}{H^2} -2\frac{\hat{g}^i(\nabla \phi, \nabla H_i)}{H_i} - 2 \phi |A|_i^2 -2 \phi Rc^i(\nu,\nu) + \phi H_i^2 + \frac{\partial \phi}{\partial t}H_i^2 d \mu. \label{Eq-lastWeakRicci}
\end{align}
Now by integrating from $a$ to $b$, $0\le a < b \le T$, and rearranging \eqref{Eq-lastWeakRicci} we find that,
\begin{align}
\int_a^b\int_{\Sigma_t}& 2\phi Rc^i(\nu,\nu)d\mu d t=  \int_{\Sigma_a} \phi H_i^2 d\mu -  \int_{\Sigma_b} \phi H_i^2 d\mu \label{WeakIntEst}
\\&+ \int_a^b\int_{\Sigma_t}2\phi\frac{|\nabla H_i|^2}{H_i^2}-2\frac{\hat{g}^i(\nabla \phi, \nabla H_i)}{H_i} +\phi(H_i^2-2|A|_i^2) + \frac{\partial \phi}{\partial t}H_i^2d\mu .
\end{align}
Notice if $m_H(\Sigma_t^i) \rightarrow 0$ then Lemma \ref{naiveEstimate} and  Proposition \ref{avgH} combined with the assumptions on $\phi$ implies,
\begin{align}
\int_a^b\int_{\Sigma_t} &\frac{\partial \phi}{\partial t}H_i^2d\mu dt \rightarrow \int_a^b \int_{\Sigma}\frac{\partial \phi}{\partial t}  \left (4+4r_0^2 e^t \right ) d\sigma dt
\\&=\int_a^b \int_{\Sigma}\frac{\partial}{\partial t}\left [\phi  \left (4+4r_0^2 e^t \right ) \right ] - 4r_0^2 e^td\sigma dt
\\&= \int_{\Sigma_b} \phi(4+r_0^2e^t)d\mu - \int_{\Sigma_a} \phi(4+r_0^2e^t)d\mu - \int_a^b \int_{\Sigma}\phi 4r_0^2 e^td\sigma dt.
\end{align} 

So by using the results of Lemma \ref{naiveEstimate} and Proposition \ref{avgH} we find that,
\begin{align}
\int_a^b\int_{\Sigma_t}& \phi Rc^i(\nu,\nu)d\mu d t \rightarrow - \int_a^b \int_{\Sigma}\phi 2r_0^2 e^td\sigma dt.
\end{align}

Notice if $m_H(\Sigma_T^i)-m_H(\Sigma_0^i) \rightarrow 0$ then,
\begin{align}
\int_a^b\int_{\Sigma_t}& \frac{\partial \phi}{\partial t}H_i^2d\mu dt \rightarrow \int_a^b\int_{\Sigma} \frac{\partial \phi}{\partial t}\left (4 - \frac{8 }{r_0}me^{-t/2} + 4 r_0^2 e^t\right )d\sigma dt 
\\&=\int_a^b\int_{\Sigma} \frac{\partial }{\partial t}\left( \phi\left (4 - \frac{8 }{r_0}me^{-t/2} + 4 r_0^2 e^t\right )\right ) -\frac{4 }{r_0}me^{-t/2} - 4r_0^2 e^t\phi d\sigma dt 
\\&= \int_{\Sigma}\phi(x,b)\left (4 - \frac{8 }{r_0}me^{-b/2} + 4 r_0^2 e^b\right ) - \phi(x,a)\left (4 - \frac{8 }{r_0}me^{-a/2} + 4 r_0^2 e^a\right )d\sigma
\\&- \int_a^b\int_{\Sigma}\left(\frac{4 }{r_0}me^{-t/2} + 4r_0^2 e^t\right)\phi d\sigma dt,
\end{align}
 by Lemma \ref{naiveEstimate} and Proposition \ref{avgH}.
  
 So by combining with \eqref{WeakIntEst} we find that , 
\begin{align}
\int_a^b\int_{\Sigma_t}& \phi Rc^i(\nu,\nu)d\mu d t \rightarrow \int_a^b\int_{\Sigma} -2\left(\frac{1 }{r_0}me^{-t/2} + r_0^2 e^t\right) \phi d\sigma dt.
\end{align}
\end{proof}

In the rest of this section we state two important estimates of IMCF without proof which were proven in the Author's other paper on stability for asymptotically flat manifolds \cite{BA2}.

\begin{Prop}[Proposition 2.9 in \cite{BA2}]\label{avgH}If $\Sigma_t^i$ is a sequence of IMCF solutions where $\int_{\Sigma_t^i} \frac{|\nabla H|^2}{H^2}d \mu \rightarrow 0$ as $i \rightarrow \infty$, $0 < H_0 \le H(x,t) \le H_1 < \infty$ and $|A|(x,t) \le A_0 < \infty$ then
\begin{align}
\int_{\Sigma_t^i} (H_i - \bar{H}_i)^2 d \mu \rightarrow 0
\end{align}
as $i \rightarrow \infty$ for almost every $t \in [0,T]$ where $\bar{H}_i = \dashint_{\Sigma_t^i}H_id \mu$. 

Let $d\mu_t^i$ be the volume form on $\Sigma$ w.r.t. $g^i(\cdot,t)$ then we can find a parameterization of $\Sigma_t$ so that 
\begin{align}
d\mu_t^i = r_0^2 e^t d\sigma
\end{align}
where $d\sigma$ is the standard volume form on the unit sphere.

Then for almost every $t \in [0,T]$ and almost every $x \in \Sigma$, with respect to $d\sigma$, we have that 
\begin{align}
H_i(x,t) - \bar{H}_i(t) \rightarrow  0,
\end{align}
along a subsequence. 

\end{Prop}

\textbf{Note:} From now on we will be using the area preserving parameterization, $F_t^i$, of the solution of IMCF, $\Sigma_t$, explained in the proof of \ref{avgH}, which is induced by an area preserving diffeomorphism between $(\Sigma, r_0^2 \sigma)$ and $(\Sigma, g^i(x,0))$. Notice that the parameterization $F_t^i$ gives a diffeomorphism from $U_T$ to an annulus portion of Euclidean space which allows us to compare $\hat{g}^i$ and $\delta$ on $\Sigma \times [0,T]$.

We end this section with an estimate for the metric of $\Sigma_t^i$ in terms of the bounds on the mean curvature and the second fundamental form.

\begin{Lem}[Lemma 2.11 in \cite{BA2}] \label{metricEst}Assume that $\Sigma_t^i$ is a solution to IMCF and let $\lambda_1^i(x,t)\le \lambda_2^i(x,t)$ be the eigenvalues of $A^i(x,t)$ then we find,
\begin{align}
 e^{\int_0^t\frac{2\lambda^i_1(x,s)}{H^i(x,s)}ds} g^i(x,0) \le g^i(x,t) &\le  e^{\int_0^t\frac{2\lambda^i_2(x,s)}{H^i(x,s)}ds} g^i(x,0).
\end{align}
\end{Lem}

\section{Convergence To A Warped Product} \label{Sect-Conv}

In this section we define the following metrics on $\Sigma \times [0,T]$,
\begin{align}
\hat{g}^i(x,t) &= \frac{1}{H^i(x,t)^2} dt^2 + g^i(x,t),
\\g_1^i(x,t)&= \frac{1}{\bar{H}^i(t)^2}dt^2 + g^i(x,t),
\\ g_2^i(x,t) &= \frac{1}{\bar{H}^i(t)^2}dt^2 + e^tg^i(x,0),
\end{align}
\begin{align}
 g_3^i(x,t) = \frac{1}{4}\left(1+ \frac{e^{-t}}{r_0^2} \right)^{-1} &dt^2 + e^t g^i(x,0)& 
\\\text{ or } \hspace{0.25cm}g_3^i(x,t) = \frac{1}{4}&\left (\frac{1}{r_0^2}e^{-t} - \frac{2}{r_0^3} m e^{-3t/2}+1 \right )^{-1} dt^2 + e^t g^i(x,0),
\\\bar{g} = \frac{1}{4}\left(1+ \frac{e^{-t}}{r_0^2} \right)^{-1}&dt^2 + r_0^2e^t \sigma
\\\text{ or } \hspace{0.25cm} g_{AdSS} = \frac{1}{4}&\left (\frac{1}{r_0^2}e^{-t}  - \frac{2}{r_0^3} m e^{-3t/2}+1 \right )^{-1} dt^2 + r_0^2e^t \sigma,
\end{align}
and successively show the pairwise convergence of the metrics in $L^2$ from $\hat{g}^i(x,t)$ to $g_3^i(x,t)$. By combining all the pairwise convergence results using the triangle inequality we will find that $\hat{g}^i -g_3^i \rightarrow 0$ in $L^2$. In the next section we will complete the desired results by showing the convergence to $\bar{g}$ or $g_{AdSS}$. 

We start by showing that $\hat{g}^i$ converges to $g_1^i$ by using Proposition \ref{avgH}.

\begin{Thm}\label{gtog1} Let $U_{T}^i \subset M_i^3$ be a sequence s.t. $U_{T}^i\subset \mathcal{M}_{r_0,H_0,I_0}^{T,H_1,A_1}$ and $m_H(\Sigma_{T}^i) \rightarrow 0$ as $i \rightarrow \infty$ or $m_H(\Sigma_{T}^i)- m_H(\Sigma_{0}^i) \rightarrow 0$ and $m_H(\Sigma_t^i)\rightarrow m > 0$. If we define the metrics, 
\begin{align}
\hat{g}^i(x,t) &= \frac{1}{H_i(x,t)^2} dt^2 + g^i(x,t),
\\g^i_1(x,t)&= \frac{1}{\overline{H}_i(t)^2}dt^2 + g^i(x,t),
\end{align}
 on $U_T^i$ then we have that,
\begin{align}
\int_{U_T^i}|\hat{g}^i -g^i_1|^2 dV \rightarrow 0 ,
\end{align}
as $i \rightarrow \infty$ where $dV$ is the volume form on $U_T^i$. 
\end{Thm}
\begin{proof}
We compute,
\begin{align}
\int_{U_T^i}|\hat{g}^i -g^i_1|^2 dV &= \int_0^T \int_{\Sigma_t^i} \frac{|\hat{g}^i -g^i_1|^2}{H} d\mu dt 
\\&=\int_0^T \int_{\Sigma_t^i} \frac{1}{H_i} \left | \frac{1}{H_i^2} - \frac{1}{\bar{H}_i^2}\right |^2d\mu dt 
\\&= \int_0^T \int_{\Sigma_t^i} \frac{|\bar{H}_i^2-H_i^2|^2}{H_i^3\bar{H}_i^2} d\mu dt 
\\&\le \frac{1}{H_0^5}\int_0^T \int_{\Sigma_t^i} |\bar{H}_i^2-H_i^2|^2 d\mu dt,  \label{Eq-1Lastgtog1}
\end{align}
where the convergence in \eqref{Eq-1Lastgtog1} follows from the pointwise convergence for almost every $t \in [0,T]$ and almost every $x \in \Sigma_t$ w.r.t $d\mu_t^{\infty}$, for a subsequence, from Proposition \ref{avgH} as well as the fact that $H_i\le H_1$ and Lebesgue's dominated convergence theorem.

We can get rid of the need for a subsequence by assuming to the contrary that for $\epsilon > 0$ there exists a subsequence so that $\int_{U_T^k}|\hat{g}^k -g^k_1|^2 dV \ge \epsilon$, but this subsequence satisfies the hypotheses of Theorem \ref{gtog1} and hence by what we have just shown we know a subsequence must converge which is a contradiction.
\end{proof}

Now we show $L^2$ convergence of $g_1^i$ to $g_2^i$ using the estimate of the metric tensor of the hypersurface $\Sigma_t$ given in Lemma \ref{metricEst}.
\begin{Thm}\label{g1tog2} Let $U_{T}^i \subset M_i^3$ be a sequence s.t. $U_{T}^i\subset \mathcal{M}_{r_0,H_0,I_0}^{T,H_1,A_1}$ and $m_H(\Sigma_{T}^i) \rightarrow 0$ as $i \rightarrow \infty$ or $m_H(\Sigma_{T}^i)- m_H(\Sigma_{0}^i) \rightarrow 0$ and $m_H(\Sigma_t^i)\rightarrow m > 0$. If we define the metrics,
\begin{align}
g^i_1(x,t)&= \frac{1}{\overline{H}_i(t)^2}dt^2 + g^i(x,t),
\\g^i_2(x,t)&= \frac{1}{\overline{H}_i(t)^2}dt^2 + e^tg^i(x,0),
\end{align}
 on $U_T^i$ then we have that,
\begin{align}
\int_{U_T^i}|g^i_1 -g^i_2|_{g_3^i}^2 dV \rightarrow 0 ,
\end{align}
as $i \rightarrow \infty$ where $dV$ is the volume form on $U_T^i$ and the norm is being calculated w.r.t. the metric $g^i_3(x,t)= \frac{r_0^2}{4}e^tdt^2 + e^tg^i(x,0)$.

Similarly, if we define,
\begin{align}
g^i_{2'}(x,t)= \frac{1}{\overline{H}_i(t)^2}dt^2 + e^{t-T}g^i(x,T),
\end{align}
 on $U_T^i$ then we have that,
\begin{align}
\int_{U_T^i}|g^i_1 -g^i_{2'}|_{g_3^i}^2 dV \rightarrow 0 ,
\end{align}
 as $i \rightarrow \infty$ where $dV$ is the volume form on $U_T^i$.
\end{Thm}
\begin{proof}
We compute,
\begin{align}
&\int_{U_T^i}|g^i_1 -g^i_2|^2 dV = \int_0^T \int_{\Sigma_t^i} \frac{|g^i_1 -g^i_2|^2}{H_i} d\mu dt 
\\&=\int_0^T \int_{\Sigma_t^i}e^{-2t}  \frac{|g^i(x,t) -e^tg^i(x,0)|_{g^i(x,0)}^2}{H_i}d\mu dt 
\\&\le \int_0^T \int_{\Sigma_t^i} e^{-2t}\frac{|g^i(x,0)|_{g^i(x,0)}^2}{H_i}\max\{|e^{\int_0^t\frac{2\lambda^i_1(x,s)}{H^i(x,s)}ds} -e^t|^2,|e^{\int_0^t\frac{2\lambda^i_2(x,s)}{H^i(x,s)}ds} -e^t|^2\} d\mu dt 
\\&\le \frac{n^2}{H_0}\int_0^T\int_{\Sigma} e^{-2t} \max\{|e^{\int_0^t\frac{2\lambda^i_1(x,s)}{H^i(x,s)}ds} -e^t|^2,|e^{\int_0^t\frac{2\lambda^i_2(x,s)}{H^i(x,s)}ds} -e^t|^2\}  d\mu dt, \label{Eq-2lastg1tog2}
\end{align}
where the convergence in \eqref{Eq-2lastg1tog2} follows from Proposition \ref{avgH} since $H_i \longrightarrow \bar{H} = 2\sqrt{\frac{e^{-t}}{r_0^2}+1}$  and $\lambda_1^i \rightarrow \lambda_2^i$ pointwise almost everywhere with respect to $d \sigma$ along a subsequence. So we have that  $\lambda_p^i \rightarrow \sqrt{\frac{e^{-t}}{r_0^2}+1}$, $p = 1,2$, for almost every $x \in \Sigma_t$ and for almost every $t \in [0,T]$ along a subsequence. This implies that $\frac{2\lambda_p^i}{H_i} \rightarrow 1$, $p = 1,2$, for almost every $x \in \Sigma_t$ and  for almost every $t \in [0,T]$ along a subsequence. Combining this with the estimate $\frac{2\lambda_p^i}{H_i} \le \frac{2A_0}{H_0}$ and Lebesgue's dominated convergence theorem we find the desired convergence above.

We can get rid of the need for a subsequence by assuming to the contrary that for $\epsilon > 0$ there exists a subsequence so that $\int_{U_T^k}|g_1^k -g^k_2|_{g_3^i}^2 dV \ge \epsilon$, but this subsequence satisfies the hypotheses of Theorem \ref{g1tog2} and hence by what we have just shown we know a subsequence must converge which is a contradiction.

We can obtain the convergence result in the case where $m_H(\Sigma_{T}^i)- m_H(\Sigma_{0}^i) \rightarrow 0$ and $m_H(\Sigma_t^i)\rightarrow m$ in a similar fasion by using the estimates of Proposition \ref{avgH} as well as Lemma \ref{GoToZero}.

 Using a similar argument, as well as the time $T$ estimate from Lemma \ref{metricEst}, we can get the second convergence result for $g_{2'}^i$.
\end{proof}

Notice that in Theorem \ref{gtog1} we were able to leverage the results of Proposition \ref{avgH} in order to gain control of the radial portion of the metric $\hat{g}^i$ as $i \rightarrow \infty$. Now we want to use the fact that we know that the average of the mean curvature is converging to that of a sphere in hyperbolic space (or ADSS) in order to complete the convergence to the warped product $g_3^i$.

\begin{Thm}\label{g2tog3} Let $U_{T}^i \subset M_i^3$ be a sequence s.t. $U_{T}^i\subset \mathcal{M}_{r_0,H_0,I_0}^{T,H_1,A_1}$ and $m_H(\Sigma_{T}^i) \rightarrow 0$ as $i \rightarrow \infty$. If we define the metrics, 
\begin{align}
g^i_2(x,t)&= \frac{1}{\bar{H}^i(t)^2}dt^2 + e^tg^i(x,0),
\\g^i_3(x,t)&= \frac{1}{4}\left(1+ \frac{e^{-t}}{r_0^2} \right)^{-1}dt^2 + e^tg^i(x,0),
\end{align}
 on $U_T^i$ then we have that,
\begin{align}
\int_{U_T^i}|g^i_2 -g^i_3|^2 dV \rightarrow 0 ,
\end{align}
as $i \rightarrow \infty$ where $dV$ is the volume form on $U_T^i$.

Instead, if $m_H(\Sigma_{T}^i)- m_H(\Sigma_{0}^i) \rightarrow 0$ and $m_H(\Sigma_t^i)\rightarrow m > 0$ and we define ,
\begin{align}
g^i_3(x,t)= \frac{1}{4}\left (\frac{1}{r_0^2}e^{-t} - \frac{2}{r_0^3} m e^{-3t/2}+1 \right )^{-1} dt^2 + e^tg^i(x,0),
\end{align} 
on $U_T^i$ then we have that,
\begin{align}
\int_{U_T^i}|g^i_2 -g^i_3|^2 dV \rightarrow 0 ,
\end{align}
as $i \rightarrow \infty$ where $dV$ is the volume form on $U_T^i$.
\end{Thm}
\begin{proof}
We calculate,
\begin{align}
&\int_{U_T^i}|\hat{g}_2^i -g^i_3|^2 dV = \int_0^T \int_{\Sigma_t^i} \frac{|\hat{g}_2^i -g^i_3|^2}{H} d\mu dt 
\\&=\int_0^T \int_{\Sigma_t^i} \frac{1}{H} \left |  \frac{1}{\bar{H}^2}- \frac{1}{4}\left(1+ \frac{e^{-t}}{r_0^2} \right)^{-1}\right |d\mu dt 
\\&= \int_0^T \int_{\Sigma_t^i} \frac{1}{4}\left(1+ \frac{e^{-t}}{r_0^2} \right)^{-1}\frac{|4\left(1+ \frac{e^{-t}}{r_0^2} \right)-\bar{H}^2|}{H\bar{H}^2} d\mu dt 
\\&\le \frac{1}{H_0^34}\int_0^T \int_{\Sigma}\left(1+ \frac{e^{-t}}{r_0^2} \right)^{-1} |4\left(1+ \frac{e^{-t}}{r_0^2} \right)-\bar{H}^2| d\mu dt, \label{Eq-last1}
\end{align}
where the convergence in \eqref{Eq-last1} follows from Lemma \ref{naiveEstimate}, \eqref{unifAvgHEst1}.

\end{proof}

\section{Convergence to Hyperbolic/Anti-deSitter Schwarzschild Space}\label{sect-End}

In this section we will complete the proofs of Theorems \ref{SPMT} and \ref{SRPI}. Note that the results of the last section are enough to prove stability in the rotationally symmetric case due to the fact that in the rotationally symmetric case we know that $(\Sigma, g^i(x,t))$ must be a round sphere by assumption. In the rotationally symmetric case the $L^2$ convergence is a weaker result than the work of Sakovich and Sormani \cite{SS} but it seems like the convergence to a warped product shown in the last section could be useful for the general case. It is work in progress with Christina Sormani to understand the relationship between $L^2$ metric convergence and intrinsic flat convergence in order to relate the results of this paper to the general conjecture stated in \cite{SS}.

In the more general case addressed by Theorems \ref{SPMT} and \ref{SRPI} we need to show that $(\Sigma, g^i(x,t))$ converges to a round sphere. In this section we will be able to show that the Gauss curvature of $\Sigma_t^i$ converges to that of a round sphere and so in order to complete the proofs of Theorems \ref{SPMT} and \ref{SRPI} we will need the following almost rigidity result of Petersen and Wei (\cite{PW}, Corollary 1.5) which allows us to go from Gauss curvature of $\Sigma_t^i$ converging to a constant to $g^i(x,t)$ converging to $r_0^2e^t\sigma(x)$ in $C^{\alpha}$.

\begin{Cor}(Petersen and Wei \cite{PW})\label{rigidity}
Given any integer $n \ge 2$, and numbers $p > n/2$, $\lambda \in \R$, $v >0$, $D < \infty$, one can find $\epsilon = \epsilon(n,p,\lambda, D) > 0$ such that a closed Riemannian $n-$manifold $(\Sigma,g)$ with
\begin{align}
&\text{vol}(\Sigma)\ge v
\\&\text{diam}(\Sigma) \le D
\\& \frac{1}{|\Sigma|} \int_{\Sigma} \|R - \lambda g \circ g\|^p d \mu \le \epsilon(n,p,\lambda,D)\label{Eq-l2curv}
\end{align}
is $C^{\alpha}$, $\alpha < 2 - \frac{n}{p}$ close to a constant curvature metric on $\Sigma$.
\end{Cor}
In our case $n=2$, $p = 2$, $\alpha < 1$ and the Riemann curvature tensor is $R = K g \circ g$, where $g \circ g$ represents the Kulkarni-Nomizu product. This shows that we need to verify that the Gauss curvature of $\Sigma_t$ is becoming constant in order to satisfy \eqref{Eq-l2curv} which is exactly what we will be able to show in the proof of the Corollaries below. Then by combining these results with the rigidity result of Petersen and Wei, Corollary \ref{rigidity}, we are able to complete the proofs of Theorems \ref{SPMT} and \ref{SRPI}. 

Now we prove Theorems \ref{SPMT} under the assumption that $K_{12}^i\ge -1$, the sectional curvature of $M_i$ tangent to $\Sigma_0^i$, for all $i$ which mimics the rotationally symmetric case where the spheres have tangent ambient sectional curvature $\ge -1$ by assumption.

\begin{Cor} \label{PKPMT}
Let $U_{T,i} \subset M_i^3$ be a sequence s.t. $U_{T,i}\subset \mathcal{M}_{r_0,H_0,I_0}^{T,H_1,A_1}$ and $m_H(\Sigma_{T}^i) \rightarrow 0$ as $i \rightarrow \infty$. If in addition we assume that $K^i_{12}(x,0) \ge -1$, the sectional curvature of $M_i^3$ tangent to $\Sigma_0$, then we have that,
\begin{align}
\hat{g}^i\rightarrow \bar{g},
\end{align}
in $L^2$ with respect to the metric $\bar{g}$.
\end{Cor}
\begin{proof}
By Lemma \ref{GoToZero} we know that $\int_{\Sigma_0^i} K^i_{12}+1 d \mu \rightarrow 0$ and if we know that $K^i_{12} \ge -1$ then we know that $K_{12}^j+1 \rightarrow 0$ pointwise a.e. on a subsequence. Combining this with the fact that $\lambda_1^j\lambda_2^j \rightarrow \frac{e^{-t}}{r_0^2}+1$ pointwise a.e. and the fact that $K^j = K_{12}^j + \lambda_1^j\lambda_2^j$ yields the desired result. Now we can apply the result of Petersen and Wei \cite{PW}, Corollary \ref{rigidity} which implies that $(\Sigma,g^i(x,0))$ is $C^{\alpha}$, $\alpha < 1$, close to a round sphere of radius $r_0$. So we can put everything together by noticing
\begin{align}
\int_{U_T} |\hat{g}^i - \bar{g}|_{\bar{g}}^2dV &\le \int_{U_T} |\hat{g}^i - \bar{g}|_{g_3^i}^2 +|(g_3^i)^{lm}(g_3^i)^{pq} -\bar{g}^{lm}\bar{g}^{pq}||\hat{g}^i - \bar{g}|_{lp} |\hat{g}^i - \bar{g}|_{mq}dV
\end{align}
where we can show the last term goes to $0$ by using that  $|g_3^i- \bar{g}|_{C^{\alpha}} \rightarrow 0$ as $i \rightarrow \infty$ and noticing that $\int_{U_T}|\hat{g}^i - \bar{g}|_{\delta}^2dV \le C$.
\end{proof}

Now we will prove Theorems \ref{SPMT} and \ref{SRPI} under the assumption of integral Ricci curvature bounds. Also, the Sobolev space $W^{1,2}(\Sigma\times [a,b])$ is defined with respect to the covariant derivative of $\delta$.

\begin{Cor} \label{End1}
Let $U_{T,i} \subset M_i^3$ be a sequence s.t. $U_{T,i}\subset \mathcal{M}_{r_0,H_0,I_0}^{T,H_1,A_1}$ and $m_H(\Sigma_{T}^i) \rightarrow 0$ as $i \rightarrow \infty$. For $[a,b]\subset [0,T]$ if we assume that, 
\begin{align}
\|Rc^i(\nu,\nu)\|_{W^{1,2}(\Sigma\times [a,b])}  \le C,
\end{align}
 and  $diam(\Sigma_t^i) \le D$ for all $i$ and $t \in [a,b]$ then,
\begin{align}
\hat{g}^i \rightarrow \bar{g},
\end{align}
in $L^2$ with respect to the metric $\bar{g}$.
\end{Cor}
\begin{proof}
By the assumption that $\|Rc^i(\nu,\nu)\|_{W^{1,2}(\Sigma\times \{0\})} \le C$ we also know that $\|Rc^i(\nu,\nu)+2\|_{W^{1,2}(\Sigma\times \{0\})} \le C$ and so by Sobolev embedding we deduce that a subsequence converges strongly in $L^2(\Sigma\times[a,b])$ to a function $k(x,t) \in L^2(\Sigma\times [a,b])$, i.e.
\begin{align}
\int_a^b\int_{\Sigma}|Rc^j(\nu,\nu)+2-k(x,t)|^2 r_0^2e^td\sigma dt \rightarrow 0. 
\end{align}
 By uniqueness of weak limits combined with Lemma \ref{WeakRicciEstimate} we find that,
 \begin{align}
 \int_a^b\int_{\Sigma}|Rc^j(\nu,\nu)+2|^2 r_0^2e^td\sigma dt  \rightarrow 0,
 \end{align} 
 in $L^2$ and hence $\int_{\Sigma}|Rc^j(\nu,\nu)+2|^2r_0^2e^td\sigma dt \rightarrow 0$ for a.e. $t \in [a,b]$ on a subsequence. So if we choose a $t' \in [0,T]$ where the pointwise convergence holds then we have that $\int_{\Sigma}(K_{12}+1)^2 r_0^2e^{t'}d\sigma  \rightarrow 0$ by noticing that,
 \begin{align}
 \int_{\Sigma}(K_{12}+1)^2 r_0^2e^{t'}d\sigma &=  \int_{\Sigma}|\frac{1}{2}(R+6) -(Rc^j(\nu,\nu)+2)|^2 r_0^2e^{t'}d\sigma
 \\&\le  \int_{\Sigma}\frac{1}{4}|(R+6)|^2 +|Rc^j(\nu,\nu)+2|^2 r_0^2e^{t'}d\sigma,
 \end{align}
 and hence, 
\begin{align}
\int_{\Sigma}(K^i-\frac{1}{r_0^2})^2 r_0^2e^{t'}d\sigma &=\int_{\Sigma}(\lambda_1^i\lambda_2^i + K_{12}^i -\frac{1}{r_0^2})^2 r_0^2e^{t'}d\sigma
\\&=\int_{\Sigma}\left((\lambda_1^i\lambda_2^i -(1+\frac{1}{r_0^2}))+ (K_{12}^i +1)\right)^2 r_0^2e^{t'}d\sigma 
\\&\le 2\int_{\Sigma}\left(\lambda_1^i\lambda_2^i -(1+\frac{1}{r_0^2})\right)^2+ \left(K_{12}^i +1\right)^2 r_0^2e^{t'}d\sigma \rightarrow 0. \label{Eq-lastEnd1}
\end{align}
This shows that $\int_{\Sigma} (K^i - \frac{1}{r_0^2})^2 r_0^2e^{t'}d\sigma \rightarrow 0$ and hence by combining with the diameter bound diam$(\Sigma_0^i)\le D$   we can apply the rigidity result of Petersen and Wei \cite{PW}, Corollary \ref{rigidity}, which implies that $|g^i(x,0) - r_0^2\sigma(x)|_{C^{\alpha}} \rightarrow 0$ as $i \rightarrow \infty$ where $\alpha < 1$. This shows that $|g_3^i- \bar{g}|_{C^{\alpha}} \rightarrow 0$ as $i \rightarrow \infty$ where $\alpha < 1$ which also implies $\int_{U_T}|\hat{g}^i - \bar{g}|_{g_3^i}dV \rightarrow 0$ as $i \rightarrow \infty$. So we can put everything together by noticing,
\begin{align}
\int_{U_T} |\hat{g}^i - \bar{g}|_{\bar{g}}^2dV &\le \int_{U_T} |\hat{g}^i - \bar{g}|_{g_3^i}^2 +|(g_3^i)^{lm}(g_3^i)^{pq} -\bar{g}^{lm}\bar{g}^{pq}||\hat{g}^i - \bar{g}|_{lp} |\hat{g}^i - \bar{g}|_{mq}dV,
\end{align}
where we can show the last term goes to $0$ by using that  $|g_3^i- \bar{g}|_{C^{\alpha}} \rightarrow 0$ as $i \rightarrow \infty$ and noticing that $\int_{U_T}|\hat{g}^i - \bar{g}|_{\delta}^2dV \le C$.

Then we can get rid of the need for a subsequence by assuming to the contrary that for $\epsilon > 0$ there exists a subsequence so that $\int_{U_T} |\hat{g}^k - \bar{g}|_{\bar{g}}^2dV \ge \epsilon$, but this subsequence satisfies the hypotheses of Theorem \ref{End1} and hence by what we have just shown we know a further subsequence must converge which is a contradiction.
\end{proof}

Now we finish up by proving a similar theorem in the Riemannian Penrose Inequality case.

\begin{Cor}\label{End2}
Let $U_{T,i} \subset M_i^3$ be a sequence s.t. $U_{T,i}\subset \mathcal{M}_{r_0,H_0,I_0}^{T,H_1,A_1}$ and $m_H(\Sigma_{T}^i)-m_H(\Sigma_0^i) \rightarrow 0$ and $m_H(\Sigma_t) \rightarrow m > 0$ as $i \rightarrow \infty$. For $[a,b]\subset [0,T]$ if we assume that, 
\begin{align}
\|Rc^i(\nu,\nu)\|_{W^{1,2}(\Sigma\times [a,b])}  \le C,
\end{align}
 and $diam(\Sigma_t^i) \le D$ for a all $i$ and $t \in [a,b]$ then 
\begin{align}
\hat{g}^i \rightarrow g_{AdSS}
\end{align}
in $L^2$ with respect to the metric $g_{AdSS}$.
\end{Cor}
\begin{proof}
The proof follows the same line of reasoning as the proof of Theorem \ref{End1}.
\end{proof}

Now we finish the paper with the proofs of Corollary \ref{PMTCOR} and Corollary \ref{RPICOR}.

\begin{Cor}\label{PMTCOR2}
Assume for all $M_i$ the smooth solution of IMCF starting at $\Sigma_0$ exists for all time, so that for all $T \in (0, \infty)$, $U_{T,i}\subset \mathcal{M}_{r_0,H_0^T,I_0}^{T,H_1,A_1}$. In addition, assume that $\displaystyle m_H(\Sigma_{\infty}^i) \rightarrow 0$ as $i \rightarrow \infty$ and that $M_i$ are uniformly asymptotically hyperbolic with respect to the IMCF coordinates then
\begin{align}
\hat{g}^i \rightarrow \bar{g}
\end{align}
on $\Sigma\times[0,T]$ in $L^2$ with respect to $\bar{g}$.
\end{Cor}
\begin{proof}
By the assumption that $M_j$ are uniformly asymptotically hyperbolic and uniformly compatible with the IMCF coordinates we can use \eqref{rtBothBigAssumption} of Definition \ref{IMCFcompatibleCoordsDef} to ensure that if we choose $t$ large enough then $r(x,t)$ will also be large. Then there exists a $T_*, \epsilon > 0$ so that for all $T \ge T_*$ we can use $\eqref{DerMetCond}$ of Definition \ref{AH} to show 
$\|Rc^i(\nu,\nu)\|_{W^{1,2}(\Sigma\times [T-\epsilon,T])} \le C$. 

Then the graph assumption \eqref{UniformGraphAssumption} of Definition \ref{IMCFcompatibleCoordsDef} implies that the metric on $\Sigma_t$ can be written as,
\begin{align}
g_{ij} &= f^2\left(\sigma_{ij} + f_if_j\right).
\end{align}
Now we let $\gamma:[0,1]\rightarrow S^2$ be the curve realizing the diameter of $g$ where $\gamma(0) = p$, $\gamma(1) = q$, $p,q \in S^2$ and we let $\alpha: [0,L]\rightarrow S^2$ be the unit speed geodesic which realizes the distance between $p$ and $q$ with respect to $\sigma$, i.e. $\alpha(0)=p$, $\alpha(L)=q$ and $d_{\sigma}(p,q) = L_{\sigma}(\alpha)$. Now we calculate,
\begin{align}
Diam(\Sigma_t) &=d_g(p,q)
\\&= L_g(\gamma) 
\\&\le L_g(\alpha)
\\&=\int_0^L \sqrt{g(\alpha',\alpha')}dt
\\&=\int_0^L\sqrt{f^2\sigma(\alpha',\alpha') + f^2 \sigma(\nabla^{S^2}f,\alpha')^2}dt
\\&\le \int_0^Lf \sqrt{\sigma(\alpha',\alpha')}dt + \int_0^L f|\sigma(\nabla^{S^2}f,\alpha')|dt 
\\&\le r_2 L + L r_2 |\nabla^{S^2}f| \le 2C_2 t \text{ Diam}(S^2) (1+C_3),
\end{align}
where we used the assumptions of Definition \ref{IMCFcompatibleCoordsDef} in the last line of the calculation.

Now we can apply the results of Corollary \ref{End1} for each fixed $T \ge T_*$ to finish the proof.
\end{proof}

\begin{Cor}\label{RPICOR2}
Assume that for all $M_i$ the smooth solution of IMCF starting at $\Sigma_0$ exists for all time, so that $U_{T,i}\subset \mathcal{M}_{r_0,H_0^T,I_0}^{T,H_1,A_1}$ for all $T \in (0, \infty)$. In addition, assume that $m_H(\Sigma_{\infty}^i)- m_H(\Sigma_{0}^i) \rightarrow 0$, $m_H(\Sigma_{\infty}^i)\rightarrow m > 0$ as $i \rightarrow \infty$, and $M_i$ are uniformly asymptotically hyperbolic with respect to the IMCF coordinates then 
\begin{align}
\hat{g}^i \rightarrow g_{AdSS}
\end{align}
on $\Sigma\times[0,T]$ in $L^2$ with respect to $g_{AdSS}$.
\end{Cor}
\begin{proof}
Use the exact same argument as in the proof of Corollary \ref{RPICOR2}.
\end{proof}


\begin{thebibliography}{99}

\bibitem{BA} B. Allen, \textit{ODE Maximum Principle at Infinity and Non-Compact Solutions of IMCF in Hyperbolic Space},  arXiv:1610.01211[math.DG], 4 Oct 2016.
\bibitem{BA2} B. Allen, \textit{IMCF and the Stability of the PMT and RPI Under $L^2$ Convergence}, \textbf{B. Allen}, Annales Henri Poincar\'{e} (2017), 1-24.

\bibitem{BA4} B. Allen, \textit{Long Time Existence of Inverse Mean Curvature Flow in Metrics Conformal to Warped Product Manifolds}, arXiv:1708.02535 [math.DG] Aug. 8, 2017.

\bibitem{B} H. Bray, \textit{Proof Of The Riemannian Penrose Inequality Using The Positive Mass Theorem}, J. Diff. Geom. \textbf{59} (2001), no.2, 177-267.

\bibitem{BF} H. Bray and F. Finster, \textit{Curvature Estimates and the Positive Mass Theorem}, Comm. Anal. Geom. \textbf{2} (2002), 291-306.

\bibitem{Br} S. Brendle, \textit{Constant Mean Curvature Surfaces in Warped Product Manifolds}, Pub. Math. de l'IHES \textbf{117} (2013), 247-269.

\bibitem{BHW} S. Brendle, P. Hung and M. Wang, \textit{A Minkowski Inequality for Hypersurfaces in the Anti-Desitter-Schwarschild Manifold} Comm. Pure Appl. Math. \textbf{44} (2016), 124-144.

\bibitem{CH} P. Chru\'{s}ciel and M. Herzlich, \textit{The Mass of Asymptotically Hyperbolic Riemannian manifolds}, Pac. J. Math., \textbf{2} (2003) 393-443.

\bibitem{C} J. Corvino, \textit{A Note on Asymptotically Flat Metrics on $\R^3$ which are Scalar-flat and Admit Minimal Spheres}, Proc. Amer. Math. Soc. \textbf{12} (2005), 3669-3678 (electronic).

\bibitem{DGS} M. Dahl, R. Gicquaud and A. Sakovich, \textit{Asymptotically Hyperbolic Manifolds with Small Mass}, Comm. Math. Phys. \textbf{325} (2014) 757-801.

\bibitem{DG} L. L. de Lima and F. Girao, \textit{An Alexandrov-Fenchel-Type Inequality in Hyperbolic Space with an Application to a Penrose Inequality}, Ann. Henri Poincare \textbf{17} (2016), 979-1002.

\bibitem{QD} Q. Ding, \textit{The inverse mean curvature flow in rotationally symmetric spaces}, Chinese Annals of Mathematics - Series B (2010), 1-18.

\bibitem{F} F. Finster, \textit{A Level Set Analysis of the Witten Spinor with Applications to Curvature Estimates},  Math. Res. Lett. \textbf{1} (2009), 41-55. 

\bibitem{FK} F. Finster and I. Kath, \textit{Curvature Estimates in Asymptotically Flat manifolds of Positive Scalar Curvature}, Comm. Anal. Geom. \textbf{5} (2002), 1017-1031. 

\bibitem{CG1} C. Gerhardt, \textit{Flow of Nonconvex Hypersurfaces into Spheres}, J. Diff. Geom. \textbf{32} (1990), 299-314.

\bibitem{CG2} C. Gerhardt, \textit{Inverse curvature flows in hyperbolic space}, J. Diff. Geom. \textbf{89}, 487 - 527, (2011)

\bibitem{HLS} L-H Huang, D. Lee, C. Sormani, \textit{Intrinsic Flat Stability of the Positive mass Theorem for Graphical Hypersurfaces of Euclidean Space}, J. fur die Riene und Ang. Math. (Crelle's Journal), \textbf{727} (2015), 1-299.

\bibitem{HI} G. Huisken and T. Ilmanen, \textit{The Inverse Mean Curvature Flow and the Riemannian Penrose Inequality}, J. Differential Geom. \textbf{59} (2001), 353-437.

\bibitem{HW} P-K. Hung and M-T. Wang, \textit{Inverse Mean Curvature Flows in the Hyperbolic 3-Space Revisited}, Calc. of Var. and PDE \textbf{54} (2015) 119-126.

\bibitem{L} D. Lee, \textit{On the Near-Equality Case of the Positive Mass Theorem}, Duke Math. J. \textbf{1} (2009), 63-80.

\bibitem{LS2} D. Lee and C. Sormani, \textit{Near-equality in the Penrose Inequality for Rotationally Symmetric Riemannian Manifolds}, Ann. Henri Poinc.,  \textbf{13} (2012), 1537-1556.

\bibitem{LS1} D. Lee and C. Sormani, \textit{Stability of the Positive Mass Theorem for Rotationally Symmetric Riemannian Manifolds}, J. fur die Riene und Ang. Math. (Crelle's Journal),  \textbf{686} (2014), 187-220.

\bibitem{LeS} P. LeFloch and C. Sormani, \textit{The Nonlinear Stability of Spaces with Low Regularity}, J. of Func. Anal. \textbf{268} (2015) no. 7, 2005-2065.

\bibitem{Li} P. Li, \textit{Geometric Analysis}, Camb. Stud. in Adv. Math., Cambridge University Press \textbf{134} (2012). 

\bibitem{AN} A. Neves, \textit{Insufficient Convergence of Inverse Mean Curvature Flow on Asymptotically Hyperbolic Manifolds}, J. Diff. Geometry \textbf{84} (2010) 191-229.

\bibitem{ANGT} A. Neves and G. Tian, \textit{Existence and Uniqueness of Constant Mean Curvature Foliations of Asymptotically Hyperbolic 3-Manifolds II}, J. fur die Riene und Ang. Math., \textbf{2010} (2010), 69-93.

\bibitem{PW} P. Petersen and G. Wei, \textit{Relative Volume Comparison with Integral Curvature Bounds}, Geom. Funct. Anal. \textbf{7} (1997), 1031-1045.

\bibitem{SS} A. Sakovich and C. Sormani, \textit{Almost Rigidity of the Positive Mass Theorem for Asymptotically Hyperbolic Manifolds with Spherical Symmetry},  arXiv:1705.07496 [math.DG], 21 May 2017.

\bibitem{S} J. Scheuer, \textit{The inverse mean curvature flow in warped cylinders of non-positive radial curvature}, Adv. in Math \textbf{306} p. 1130-1163 (2017).

\bibitem{S2} J. Scheuer, \textit{Inverse Curvature Flows in Riemannian Warped Products}, arXiv:1712.09521 [math.DG], 27 Dec 2017.

\bibitem{U} J. Urbas, \textit{On the Expansion of Starshaped Hypersurfaces by Symmetric Functions of their Principal Curvatures}, Math. A. \textbf{205} (1990), 355-372.

\bibitem{W} X. Wang, \textit{The Mass of Asymptotically Hyperbolic Manifolds}, J. Diff. Geom., \textbf{2} (2001) 273-299.

\bibitem{Z} H. Zhou, \textit{Inverse Mean Curvature Flows in Warped Product Manifolds}, J. Geom. Anal.  (2017), 1-24.
\end{thebibliography}
\end{document}